\documentclass[10pt, a4paper]{article}
\usepackage{amsfonts}
\usepackage{amssymb,amsthm, amsmath,graphicx}
\usepackage{bbm}
\usepackage{url}

\newtheorem{theorem}{Theorem}

\newtheorem{proposition}[theorem]{Proposition}
\newtheorem{corollary}[theorem]{Corollary}
\newtheorem{lemma}[theorem]{Lemma}
\newtheorem{construction}[theorem]{Construction}

\newtheorem{example}[theorem]{Example}

\def\CaM{\mathcal{M}}
\def\CaC{\mathcal{C}}
\def\k{\mathbbmss{k}}
\def\C{\mathbb{C}}
\def\N{\mathbb{N}}
\def\Z{\mathbb{Z}}

\title{Combinatorial properties and characterization of glued semigroups}

\author{J. I. Garc\'{\i}a-Garc\'{\i}a\footnote{Departamento de Matem\'aticas, Universidad de C\'adiz,
E-11510 Puerto Real (C\'{a}diz, Spain). E-mail: ignacio.garcia@uca.es. Partially supported by MTM2010-15595 and Junta de Andaluc\'{\i}a group FQM-366. }\\
M.A. Moreno-Fr\'{\i}as\footnote{Departamento de Matem\'aticas, Universidad de C\'adiz,
E-11510 Puerto Real (C\'{a}diz, Spain). E-mail: mariangeles.moreno@uca.es. Partially supported by MTM2008-06201-C02-02 and Junta de Andaluc\'{\i}a group FQM-298.}\\
A. Vigneron-Tenorio\footnote{Departamento de Matem\'aticas, Universidad de C\'adiz,
E-11405 Jerez de la Frontera (C\'{a}diz, Spain). E-mail: alberto.vigneron@uca.es. Partially supported by the grant MTM2007-64704 (with the help of FEDER Program), MTM2012-36917-C03-01 and Junta de Andaluc\'{\i}a group FQM-366.}\\
}

\begin{document}

\date{}

\maketitle

\begin{abstract}
This work focuses on the combinatorial properties of glued semigroups and provides its combinatorial characterization. Some classical results for affine glued semigroups are generalized and some methods to obtain glued semigroups are developed.
\smallskip

{\small \emph{Keywords:} Gluing of semigroups, semigroup, semigroup ideal, simplicial complex, toric ideal.}

\smallskip
{\small \emph{MSC-class:} 20M14 (Primary), 20M05 (Secondary).}
\end{abstract}

\section*{Introduction}

Let $S=\langle n_1,\ldots ,n_l\rangle$ be a finitely generated commutative semigroup with zero element which is  reduced (i.e. $S\cap (-S)=\{0\}$) and cancellative (if $m,n,n'\in S$ and $m+n=m+n'$ then $n=n'$).
Under these settings if $S$ is torsion-free then it is isomorphic to a subsemigroup of $\N^p$ which means it
is an affine semigroup (see \cite{Rosales3}). From now on assume that all the semigroups appearing in this work are finitely generated, commutative, reduced and cancellative, but not necessarily torsion-free.

Let $\k$ be a field and $\k[X_1,\ldots ,X_l]$ the polynomial ring in $l$ indeterminates. This polynomial ring is obviously an $S-$graded ring (by assigning the $S$-degree $n_i$ to the indeterminate $X_i,$ the $S$-degree of $X^\alpha =X_1^{\alpha_1}\cdots X_l^{\alpha_l}$ is $\sum _{i=1}^l \alpha _i n_i\in S$).
It is well known that the ideal $I_S$ generated by
$$\left\{X^\alpha -X^\beta | \sum_{i=1}^l \alpha_in_i=\sum_{i=1}^l \beta_in_i\right\}\subset \k[X_1,\ldots ,X_l]$$
is an $S-$homogeneous binomial ideal called {\em semigroup ideal} (see \cite{Herzog70} for details). If $S$ is torsion-free, the ideal obtained defines a toric variety (see \cite{Sturmfels95} and the references therein). By Nakayama's lemma, all the minimal generating sets of $I_S$ have the same cardinality and the $S-$degrees of its elements can be determinated.

The main goal of this work is to study the semigroups which result from the gluing of others two. This concept was introduced by Rosales in \cite{Rosales1} and it is closely related with complete intersection ideals (see \cite{Thoma} and the references therein).
A semigroup $S$ minimally generated by $A_1\sqcup A_2$ (with $A_1=\{n_1,\ldots ,n_r\}$ and $A_2=\{n_{r+1},\ldots ,n_l \}$) is the gluing of $S_1=\langle A_1\rangle $ and $S_2=\langle A_2\rangle $
if there exists a set of generators $\rho$ of $I_S$ of the form
$\rho=\rho_1\cup \rho_2 \cup \{X^{\gamma} -X^{\gamma '}\},$
where $\rho_1,\rho_2$ are generating sets of
$I_{S_1}$ and $I_{S_2}$ respectively,
$X^{\gamma } -X^{\gamma '}\in I_S$ and
the supports of $\gamma$ and $\gamma'$ verify
$\mbox{supp } (\gamma )\subset\{1,\ldots ,r\}$ and $\mbox{supp } (\gamma ')\subset \{r+1,\ldots ,l\}.$
Equivalently, $S$ is the gluing of $S_1$ and $S_2$ if $I_S=I_{S_1}+I_{S_2}+\langle X^{\gamma } -X^{\gamma '}\rangle.$
A semigroup is a \emph{glued semigroup} when it is the gluing of others two.

As seen, glued semigroups can be determinated by the minimal generating sets of $I_S$ which can be studied by using combinatorial methods from certain simplicial complexes (see \cite{BCMP}, \cite{Eliahou} and \cite{OjVi}).
In this work the simplicial complexes used are defined as follows:
for any $m \in S,$  set
\begin{equation}\label{vertices}C_m =
\{X^\alpha= X_1^{\alpha _1} \cdots X_l^{\alpha_l} \mid \sum_{i=1}^l \alpha_i n_i= m\}\end{equation}
and the simplicial complex
\begin{equation}\label{nabla}\nabla_m = \{ F \subseteq C_m \mid \gcd(F) \neq 1 \},\end{equation} with $\gcd(F)$  the {\em greatest common divisor} of the monomials in $F.$

Furthermore, some methods which require  linear algebra and integer programming are given
to obtain examples of glued semigroups.

The content of this work is organized as follows. Section \ref{generalizaciones_gluing} presents the tools to generalize to non torsion-free semigroups a classical  characterization of affine gluing semigroups  (Proposition \ref{proposicion_gluing_gupos}). In Section \ref{section_gluing_comb}, the non-connected simplicial complexes $\nabla _m$  associated to glued semigroups are studied. By using the vertices of the connected components of these complexes we give a combinatorial characterization of glued semigroups as well as their glued degrees (Theorem \ref{main_gluing}). Besides, in Corollary \ref{corolario_indispensable}
we deduce the  conditions  in order  the ideal of a glued semigroup to be uniquely generated.
Finally, Section \ref{seccion_generando_gluing} is devoted to the construction of glued semigroups (Corollary \ref{gamma_afin}) and affine glued semigroups (Subsection \ref{section_affine}).

\section{Preliminaries and generalizations about glued semigroups}\label{generalizaciones_gluing}

A binomial of $I_S$ is called {\em indispensable} if it is an element of all system of generators of $I_S$ (up to a scalar multiple). This kind of binomials were introduced in \cite{Ohsugi} and they have  an important role in Algebraic Statistics. In \cite{OjVi2} the authors characterize  indispensable binomials by using  simplicial complexes $\nabla _m.$ Note that if $I_S$ is generated by its indispensable binomials, it is uniquely generated up to a scalar multiples.

With the above notation, the semigroup $S$ is associated to the lattice $\ker S$ formed by the elements $\alpha=(\alpha _1,\ldots ,\alpha _l)\in\Z^l$ such that $\sum_{i=1}^l \alpha_in_i =0$.
Given $G$ a system of generators of $I_S$, the lattice $\ker S$ is generated by the elements $\alpha-\beta$ with $X^\alpha -X^\beta\in G$ and $\ker S$ also verifies that $\ker S \cap \N ^l=\{0\}$ if and only if $S$ is reduced. If $\CaM(I_S)$ is a minimal generating set of $I_S,$  denote by $\CaM(I_S)_m\subset \CaM(I_S)$  the set of elements whose $S-$degree is equal to  $m\in S$ and by $Betti(S)$ the set of the $S-$degrees of the elements of $\CaM(I_S).$ When $I_S$ is minimally generated by $\mbox{rank} (\ker S)$ elements, the semigroup
$S$ is called a {\em complete intersection} semigroup.

Let $\CaC(\nabla_m)$ be the number of connected components of $\nabla_m$. The cardinality of $\CaM(I_S)_m$ is equal to $\CaC(\nabla_m)-1$ (see Remark 2.6 in \cite{BCMP} and Theorem 3 and Corollary 4 in \cite{OjVi}) and the complexes associated to the elements in $Betti(S)$ are non-connected.

\begin{construction}{{\em (\cite[Proposition 1]{Eliahou}).}}\label{construction1}
For each $m\in Betti(S)$ the set $\CaM (I_S)_m$ is obtained by taking $\CaC(\nabla _m)-1$ binomials with monomials in different connected components of $\nabla _m$
satisfying that two different binomials have not their corresponding monomials in the same components
 and fulfilling that there is at least a monomial of every connected component of $\nabla _m$.
 This let us construct a minimal generating set of $I_S$ in a combinatorial way.
\end{construction}

 Let $S$ be minimally\footnote{We consider a minimal generator set of $S$, in the other case $S$ is trivially the  gluing of the semigroup generated by one of its non minimal generators and the semigroup generated by the others.} generated by $A_1\sqcup A_2$ with  $A_1=\{a_1,\ldots ,a_r\}$ and $A_2=\{b_1,\ldots ,b_t\}.$ From now on,  identify the sets $A_1$ and $A_2$ with the  matrices $(a_1 | \cdots | a_r)$ and $(b_1 | \cdots | b_t).$
Denote by $\k [A_1]$ and $\k [A_2]$  the polynomial rings $\k [X_1,\ldots ,X_r]$ and $\k [Y_1,\ldots ,Y_t],$ respectively. A monomial is a {\em pure monomial} if it has indeterminates only in $X_1,\ldots ,X_r$ or only in $Y_1,\ldots ,Y_t$, otherwise  it is a {\em mixed monomial}.
If $S$ is the gluing of $S_1=\langle A_1 \rangle$ and $S_2=\langle A_2\rangle,$ then
the binomial $X^{\gamma _X} -Y^{\gamma _Y} \in I_S$ is a {\em glued binomial} if $\CaM(I_{S_1})\cup \CaM(I_{S_2}) \cup \{X^{\gamma _X} -Y^{\gamma _Y}\}$ is a generating set of $I_S$ and  in this case the element $d=S\mbox{-}degree(X^{\gamma _X})\in S$ is called a {\em glued degree}.

It is clear that if $S$ is a glued semigroup, the lattice $\ker S$ has a basis of the form \begin{equation}\label{grupo}\{L_1,L_2,(\gamma _X,-\gamma _Y)\}\subset \Z ^{r+t},\end{equation}
where the supports of the elements in $L_1$ are in $\{1,\ldots ,r\},$ the supports of the elements in $L_2$ are in $\{r+1,\ldots ,r+t\},$ $\ker S_i =\langle L_i\rangle\, (i=1,2)$ by considering only the coordinates in $\{1,\ldots ,r\}$ or $\{r+1,\ldots ,r+t\}$ of $L_i,$ and $(\gamma _X,\gamma _Y)\in \N ^{r+t}.$ Moreover, since $S$ is reduced, one has that $\langle L_1\rangle \cap \N ^{r+t}=\langle L_2\rangle \cap \N ^{r+t}=\{0\}.$ Denote by  $\{\rho_{1i}\}_i$  the elements in $L_1$ and by $\{\rho_{2i}\}_i$  the elements in $L_2.$

The following Proposition generalizes \cite[Theorem 1.4]{Rosales1} to non-torsion free semigroups.

\begin{proposition}\label{proposicion_gluing_gupos}
The semigroup
$S$ is the gluing of $S_1$ and $S_2$ if and only if there exists $d\in (S_1\cap S_2)\setminus\{0\}$ such that $G(S_1)\cap G(S_2)=d\Z,$ where $G(S_1),$ $G(S_2)$ and $d\Z$ are the associated commutative groups of $S_1,$ $S_2$ and $\{d\}$, respectively.
\end{proposition}

\begin{proof}
Assume that $S$ is the gluing of $S_1$ and $S_2.$ In this case, $\ker S$ is generated by the set (\ref{grupo}). Since $(\gamma _X,-\gamma _Y)\in \ker S,$ the element $d$ is equal to $ A_1\gamma _X=A_2\gamma _Y\in S$ and $d\in S_1 \cap S_2\subset G(S_1)\cap G(S_2).$
Let $d'$ be in $G(S_1)\cap G(S_2),$ then there exists $(\delta _1,\delta _2)\in \Z^r\times \Z^t$ such that $d'=A_1\delta _1=A_2\delta _2.$ Therefore $(\delta _1,-\delta _2)\in \ker S$ because $(A_1 | A_2)(\delta _1,-\delta _2)=0$ and so there exist $\lambda,\lambda_i^{\rho_{1}},\lambda_i^{\rho_{2}}\in \Z$ satisfying
$$\left\{\begin{array}{ccc}
(\delta _1,0) & = & \sum _i \lambda_i^{\rho_{1}} \rho _{1i} + \lambda (\gamma _X,0) \\ \\
(0,\delta _2) & = & -\sum _i \lambda_i^{\rho_{2}} \rho _{2i} + \lambda (0,\gamma _Y) ,
\end{array}\right.$$
and $d'= A_1\delta _1= \sum _i \lambda_i^{\rho_{1}} (A_1|0)\rho _{1i} + \lambda A_1 \gamma _X= \lambda d.$ We conclude that $G(S_1)\cap G(S_2)=d\Z$ with $d\in S_1\cap S_2.$

Conversely, suppose there exists $d\in (S_1\cap S_2)\setminus\{0\}$ such that $G(S_1)\cap G(S_2)=d\Z.$ We see  $I_S=I_{S_1}+I_{S_2}+\langle X^{\gamma _X}- Y^{\gamma _Y}\rangle.$
Trivially, $I_{S_1}+I_{S_2}+\langle X^{\gamma _X}- Y^{\gamma _Y}\rangle \subset I_S.$
Let $X^\alpha Y^\beta -X^\gamma Y^\delta$ be a binomial in $I_S.$ Its $S-$degree is $A_1\alpha + A_2 \beta = A_1 \gamma + A_2\delta.$ Using  $A_1(\alpha -\gamma)=A_2(\beta -\delta)\in G(S_1)\cap G(S_2)= d\Z,$
there exists  $\lambda \in \Z$ such that
$A_1\alpha  =  A_1\gamma +\lambda d$ and $A_2\delta  =  A_2\beta +\lambda d.$
We have the following cases:
\begin{itemize}
\item If $\lambda =0,$
$$X^\alpha Y^\beta -X^\gamma Y^\delta= X^\alpha Y^\beta -X^\gamma Y^\beta +X^\gamma Y^\beta -X^\gamma Y^\delta=$$ $$=Y^\beta (X^\alpha -X^\gamma)+X^\gamma(Y^\beta-Y^\delta)\in I_{S_1}+I_{S_2}.$$

\item If $\lambda >0,$
$$X^\alpha Y^\beta -X^\gamma Y^\delta=$$ $$= X^\alpha Y^\beta -X^\gamma X^{\lambda \gamma _X}Y^\beta +X^\gamma X^{\lambda \gamma _X}Y^\beta -X^\gamma X^{\lambda \gamma _Y}Y^\beta +X^\gamma X^{\lambda \gamma _Y}Y^\beta -X^\gamma Y^\delta=$$ $$=Y^\beta (X^\alpha -X^\gamma X^{\lambda \gamma _X})+ X^\gamma Y^\beta (X^{\lambda \gamma _X}-Y^{\lambda \gamma _Y}) + X^\gamma(Y^{\lambda \gamma _Y} Y^\beta-Y^\delta).$$
Using that
$$X^{\lambda \gamma _X}-Y^{\lambda \gamma _Y}= (X^{\gamma _X}-Y^{\gamma _Y})\left(\sum _{i=0}^{\lambda -1} X^{(\lambda-1 -i)\gamma _X}Y^{i\gamma _Y}\right),$$
 the binomial $X^\alpha Y^\beta -X^\gamma Y^\delta$  belongs to $I_{S_1}+I_{S_2}+\langle X^{\gamma _X}- Y^{\gamma _Y}\rangle$.

\item The case $\lambda <0$ is solved similarly.
\end{itemize}

We conclude that $I_S=I_{S_1}+I_{S_2}+\langle X^{\gamma _X}- Y^{\gamma _Y}\rangle.$
\end{proof}

From above proof it is deduced that given the partition of the system of generators of $S$ the glued degree is unique.

\section{Glued semigroups and combinatorics}\label{section_gluing_comb}

Glued semigroups by means of  non-connected simplicial complexes are characterized.
For any $m \in S,$  redefine $C_m$ from (\ref{vertices}), as $$C_m =
\{X^\alpha Y^\beta= X_1^{\alpha _1} \cdots X_r^{\alpha_r} Y_1^{\beta _1} \cdots Y_t^{\beta_t} \mid \sum_{i=1}^r \alpha_i a_i + \sum_{i=1}^t \beta _i b_i =
m\}$$ and consider the sets of vertices  and the simplicial complexes
$$C_m^{A_1} =
\{X_1^{\alpha _1} \cdots X_r^{\alpha_r} \mid \sum_{i=1}^r \alpha_i a_i =
m\},\, \nabla_m^{A_1} = \{ F \subseteq C_m^{A_1} \mid \gcd(F) \neq 1 \},$$
$$C_m^{A_2} =
\{Y_1^{\beta _1} \cdots Y_t^{\beta_t} \mid \sum_{i=1}^t \beta _i b_i =
m\},\, \nabla_m^{A_2} = \{ F \subseteq C_m^{A_2} \mid \gcd(F) \neq 1 \},$$where $A_1=\{a_1,\ldots ,a_r\}$ and $A_2=\{b_1,\ldots ,b_t\}$ as in Section \ref{generalizaciones_gluing}.
Trivially, the relations between $\nabla_m^{A_1}, \nabla_m^{A_2}$ and $\nabla_m$ are
\begin{equation}\label{nablas}
\nabla_m^{A_1} = \{ F\in \nabla _m | F\subset C_m^{A_1}\},\,
\nabla_m^{A_2} = \{ F\in \nabla _m | F \subset C_m^{A_2}\}.
\end{equation}

The following result shows an important property of the simplicial complexes associated to glued semigroups.

\begin{lemma}\label{MixComp}
Let $S$ be the gluing of $S_1$ and $S_2,$ and $m\in Betti(S).$
Then all the connected components of $\nabla _m$
have at least a pure monomial. In addition, all mixed monomials of $\nabla _m$ are in the same connected component.
\end{lemma}

\begin{proof}
Suppose that there exists $C,$ a connected component of $\nabla_m$ only with mixed monomials. By Construction \ref{construction1} in all generating sets of $I_S$ there is at least a binomial with a mixed monomial, but this does not occur in  $\CaM(I_{S_1})\cup \CaM(I_{S_2}) \cup \{X^{\gamma _X} -Y^{\gamma _Y}\}$ with $X^{\gamma _X} -Y^{\gamma _Y}$ a glued binomial.

Since $S$ is a glued semigroup, $\ker S$ has a system of generators as  (\ref{grupo}). Let $X^\alpha Y^\beta, X^\gamma Y^\delta\in C_m$ be two monomials such that $\gcd(X^\alpha Y^\beta, X^\gamma Y^\delta)=1.$ In this case, $(\alpha,\beta)-(\gamma,\delta)\in \ker S$ and there exist $\lambda,\lambda_i^{\rho_{1}},\lambda_i^{\rho_{2}}\in \Z$ satisfying:
$$\left\{\begin{array}{ccc}
(\alpha -\gamma,0) & = & \sum _i \lambda_i^{\rho_{1}} \rho _{1i} + \lambda (\gamma _X,0) \\ \\
(0,\beta - \delta) & = & \sum _i \lambda_i^{\rho_{2}} \rho _{2i} - \lambda (0,\gamma _Y)
\end{array}\right.$$
\begin{itemize}
\item If $\lambda =0,$ $\alpha -\gamma \in \ker S_1$ and  $\beta -\delta \in \ker S_2$, then $A_1\alpha = A_1 \gamma,$ $A_2 \beta = A_2\delta$ and $X^\alpha Y^\delta \in C_m.$
\item If $\lambda >0,$ $(\alpha,0)  = \sum _i \lambda_i^{\rho_{1}} \rho _{1i} + \lambda (\gamma _X,0) +(\gamma,0)$ and $$ A_1\alpha  = \sum _i \lambda_i^{\rho_{1}} (A_1|0)\rho _{1i} + \lambda A_1\gamma _X +A_1\gamma= \lambda d +A_1\gamma,$$ then $X^{\lambda \gamma _X}X^\gamma Y^\beta \in C_m.$

\item The case $\lambda <0$ is solved likewise.
\end{itemize}
In any case, $X^\alpha Y^\beta$ and $X^\gamma Y^\delta$ are in the same connected component of $\nabla _m.$
\end{proof}

We now describe the simplicial complexes that correspond to the $S-$degrees which are multiples of the glued degree.

\begin{lemma}\label{multiplos_d}
Let $S$ be the gluing of $S_1$ and $S_2,$ $d\in S$ the glued degree and $d'\in S\setminus\{d\}.$
Then $C_{d'}^{A_1}\neq \emptyset \neq C_{d'}^{A_2}$ if and only if $d'\in (d\N)\setminus\{0\}.$ Furthermore, the simplicial complex $\nabla _{d'}$ has at least one connected component with elements in $C_{d'}^{A_1}$ and $C_{d'}^{A_2}.$
\end{lemma}

\begin{proof}
If  there exist $X^\alpha ,Y^\beta \in C_{d'},$ then $d'=\sum _{i=1}^r \alpha _i a_i = \sum _{i=1}^t \beta _i b_i\in S_1 \cap S_2 \subset  G(S_1)\cap G(S_2)=d\Z.$ Hence, $d'\in d\N.$

Conversely, let $d'=jd$ with $j\in \N$ and  $j>1$ and take $X^{\gamma _X}-Y^{\gamma _Y} \in I_S$ be a glued binomial. It is easy to see that $X^{j\gamma _X} ,Y^{j\gamma _Y} \in C_{d'}$ and thus $\{X^{j\gamma _X}, X^{(j-1)\gamma _X} Y^{\gamma _Y}\}$ and $\{ X^{(j-1)\gamma _X} Y^{\gamma _Y},Y^{j\gamma _Y}\}$ belong to $\nabla_{d'}.$
\end{proof}

The following Lemma is a combinatorial version of \cite[Lemma 9]{GarciaOjeda} and it is a necessary condition of Theorem \ref{main_gluing}.

\begin{lemma}\label{without_mix}
Let $S$ be the gluing of $S_1$ and $S_2,$ and $d\in S$ the glued degree.
Then the elements of $C_{d}$ are pure monomials and $d\in Betti(S).$
\end{lemma}

\begin{proof}
The order $\preceq_S$ defined by $m' \preceq_S m$ if $m-m'\in S$ is a partial order on $S.$

Assume  there exists a mixed monomial $T\in C_d.$ By Lemma \ref{MixComp}, there exists a pure monomial $Y^b$ in $C_d$  such that $\{T,Y^b\}\in \nabla _d$ (the proof is analogous if we consider $X^a$ with $\{T,X^a\}\in \nabla _d$). Now take $T_1=\gcd (T,Y^b)^{-1}T$ and $Y^{b_1}=\gcd (T,Y^b)^{-1}Y^b$. Both monomials are in $C_{d'},$ where $d'$ is equal to $d$ minus the $S-$degree of $\gcd (T,Y^b).$
By Lemma \ref{multiplos_d}, if $C_{d'}^{A_1}\neq \emptyset$ then  $d'\in d\N$, but since
 $d' \prec_S d$ this
is not possible.
So, if $T_1$ is a mixed monomial and $C_{d'}^{A_1} = \emptyset$, then $C_{d'}^{A_2}\neq \emptyset$.
If there exists a pure monomial in $C_{d'}^{A_2}$ connected to a mixed monomial in $C_{d'},$ we perform the same process  obtaining  $T_2,Y^{b_2}\in C_{d''}$ with $T_2$ a mixed monomial and $d''\prec_S d'$.
This process can be repeated if there exist a pure monomial and a mixed monomial in the same connected component. By degree reasons this cannot be performing indefinitely,  an element $d^{(i)}\in Betti(S)$ verifying that $\nabla _{d^{(i)}}$ is not connected
having a connected component with only  mixed monomials is found. This contradicts Lemma \ref{MixComp}.
\end{proof}

After examining the structure of the simplicial complexes associated to glued semigroups, we enunciate a combinatorial characterization  by means of the non-connected simplicial complexes $\nabla _m.$

\begin{theorem}\label{main_gluing}
The semigroup
$S$ is the gluing of $S_1$ and $S_2$ if and only if the following conditions are fulfilled:
\begin{enumerate}
\item For all $d'\in Betti(S),$ any connected component of $\nabla _{d'}$
has at least a pure monomial.
\item There exists a unique $d\in Betti(S)$ such that $C_d^{A_1}\neq \emptyset \neq C_d^{A_2}$ and the elements in $C_{d}$ are pure monomials.
\item For all $d'\in Betti(S)\setminus \{d\}$ with $C_{d'}^{A_1}\neq \emptyset \neq C_{d'}^{A_2},$ $d'\in d\N.$
\end{enumerate}

Besides, the above $d\in Betti(S)$ is the glued degree.
\end{theorem}

\begin{proof}
If $S$ is the gluing of $S_1$ and $S_2,$ the result is obtained from Lemmas \ref{MixComp}, \ref{multiplos_d} and \ref{without_mix}.

Conversely, by  hypothesis 1 and 3, given $d'\in Betti(S)\setminus\{d\}$  the set $\CaM (I_{S_1})_{d'}$
is constructed from $C_{d'}^{A_1}$ and $\CaM (I_{S_2})_{d'}$ from  $C_{d'}^{A_2}$ as in Construction \ref{construction1}.
Analogously, if  $d\in Betti(S),$  the set $\CaM (I_{S})_{d}$ is obtained from the union of $\CaM (I_{S_1})_{d}$, $\CaM (I_{S_2})_{d}$ and the binomial $X^{\gamma _X} -Y^{\gamma _Y}$ with $X^{\gamma _X} \in C_d^{A_1}$ and $Y^{\gamma _Y} \in C_d^{A_2}.$
Finally $$\displaystyle{\bigsqcup _{m\in Betti(S)} }\Big(\CaM(I_{S_1})_m\sqcup \CaM(I_{S_2})_m \Big) \sqcup \{X^{\gamma _X} -Y^{\gamma _Y}\}$$ is a generating set of $I_S$ and  $S$ is the gluing of $S_1$ and $S_2.$
\end{proof}

From Theorem \ref{main_gluing} we obtain an equivalent property to Theorem 12 in \cite{GarciaOjeda}
by using the {\em language} of monomials and binomials.

\begin{corollary}\label{corolario_indispensable}
Let $S$ be the gluing of $S_1$ and $S_2,$ and $X^{\gamma _X} -Y^{\gamma _Y} \in I_S$ a glued binomial with $S-$degree $d.$ The ideal $I_S$ is minimally generated by its indispensable binomials if and only if the following conditions are fulfilled:
\begin{itemize}
\item The ideals $I_{S_1}$ and $I_{S_2}$ are minimally generated by their indispensable binomials.
\item The element $X^{\gamma _X} -Y^{\gamma _Y}$ is an indispensable binomial of $I_S.$
\item For all $d'\in Betti(S),$ the elements of $C_{d'}$ are pure monomials.
\end{itemize}
\end{corollary}

\begin{proof}
Suppose that $I_S$ is generated by its indispensable binomials. By \cite[Corollary 6]{OjVi2},  for all $ m\in Betti(S)$ the simplicial complex $\nabla _m$ has only two vertices. By  Contruction \ref{construction1}
$\nabla_d=\{\{X^{\gamma _X}\},\{Y^{\gamma _Y}\}\}$ and by Theorem \ref{main_gluing} for all $d'\in Betti(S)\setminus\{d\}$
the simplicial $\nabla _{d'}$ is equal to $\nabla_{d'}^{A_1}$ or $\nabla_{d'}^{A_2}$.
In any case, $X^{\gamma _X} -Y^{\gamma _Y}\in I_S$ is an indispensable binomial, and $I_{S_1}, I_{S_2}$ are generated by their indispensable binomials.

Conversely,  suppose that $I_S$ is not generated by its indispensable binomials. Then, there exists $d'\in Betti(S)\setminus \{d\}$ such that $\nabla _{d'}$ has more than two vertices in at least two different connected components. By hypothesis, there are not mixed monomials in $\nabla _{d'}$ and thus:
\begin{itemize}
\item If $\nabla _{d'}$ is equal to $\nabla_{d'}^{A_1}$ (or $\nabla_{d'}^{A_2}$), then $I_{S_1}$ (or $I_{S_2}$) is not generated by its indispensable binomials.
\item Otherwise, $C_{d'}^{A_1}\neq \emptyset \neq C_{d'}^{A_2}$ and by Lemma \ref{multiplos_d}, $d'=jd$ with $j\in \N,$ therefore $X^{(j-1)\gamma _X} Y^{\gamma _Y}\in C_{d'}$ which contradicts the hypothesis.
\end{itemize}
We conclude  $I_S$ is generated by its indispensable binomials.
\end{proof}

The following example taken from \cite{Thoma} illustrates the above results.

\begin{example}
Let $S\subset \N ^2$ be the semigroup generated by the set $$\left\{ (13, 0), (5, 8), (2, 11), (0, 13), (4, 4), (6, 6), (7, 7), (9, 9) \right\}.$$
In this case, $Betti(S)$ is $$\{(15,15),(14,14),(12,12),(18,18),(10,55),(15,24),(13,52),(13,13)\}.$$

Using the appropriated notation for the indeterminates in the polynomial ring $\k[x_1,\ldots ,x_4,y_1,\ldots ,y_4]$ ($x_1,x_2,x_3$ and $x_4$ for the first four generators of $S$ and $y_1,y_2,y_3,y_4$ for the others),  the simplicial complexes associated to the elements in $Betti(S)$ are those that appear in Table \ref{figura1}.
\begin{table}[h]
{\footnotesize $$\begin{array}{||c|c|c|c||}\cline{1-4} & & & \\
C_{(15,15)}=\{{y_1^2y_3},{y_2y_4} \} & C_{(14,14)}=\{{y_1^2y_2},{y_3^2} \} & C_{(12,12)}=\{{y_1^3},{y_2^2} \} & C_{(10,55)}=\{{x_1^2x_4},{x_3^5} \}
\\ & & &\\
\nabla_{(15,15)} & \nabla_{(14,14)} & \nabla_{(12,12)}& \nabla_{(10,55)}
\\ & & & \\
\includegraphics[scale=0.5]{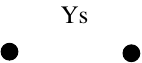} & \includegraphics[scale=0.5]{basica1_y1}  & \includegraphics[scale=0.5]{basica1_y1} & \includegraphics[scale=0.5]{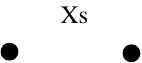}  \\
\cline{1-4} & & & \\
C_{(18,18)}=\{{y_1^2y_2},{y_2^3}, & C_{(15,24)}=\{{x_1x_2x_3},{x_2^3}, & C_{(13,52)}=\{{x_2x_3^4},{x_1x_4^4}, & C_{(13,13)}=\{{x_1x_4},{y_1y_4},
\\
{y_1y_3^2},{y_4^2} \} &
{x_3y_1y_4},{x_3y_2y_3} \} &
{x_4^3y_1y_4},{x_4^3y_2y_3}\} & {y_2y_3} \}
\\ 
\nabla_{(18,18)} & \nabla_{(15,24)} & \nabla_{(13,52)}& \nabla_{(13,13)}
\\ 
\includegraphics[scale=0.5]{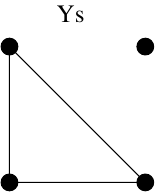} & \includegraphics[scale=0.5]{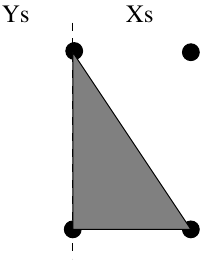}  & \includegraphics[scale=0.5]{basica4_11} & \includegraphics[scale=0.5]{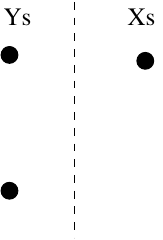} \\ \cline{1-4}
\end{array}$$}
\caption{Non-connected simplicial complexes associated to $Betti(S).$}\label{figura1}
\end{table}
From this table  and by using Theorem \ref{main_gluing}, the semigroup $S$ is the gluing of  $\langle (13, 0), (5, 8), (2, 11), (0, 13)\rangle$ and $\langle (4, 4), (6, 6), (7, 7), (9, 9)\rangle$ and the glued degree is $(13,13).$
From Corollary \ref{corolario_indispensable}, the ideal $I_S$ is not generated by its indispensable binomials ($I_S$ has only four indispensable binomials).
\end{example}

\section{Generating glued semigroups}\label{seccion_generando_gluing}

In this section,  an algorithm to obtain examples of glued semigroups is given.
Consider
$A_1=\{a_1,\ldots ,a_r\}$ and  $A_2=\{b_1,\ldots ,b_t\}$ two minimal generator sets of the  semigroups
$T_1$ and $T_2$
and  let $L_j=\{\rho_{ji}\}_i$ be a basis of $\ker T_j$ with $j=1,2.$
Assume that $I_{T_1}$ and $I_{T_2}$ are nontrivial proper ideals of their corresponding polynomial rings.
Consider $\gamma _X$ and $\gamma _Y$ be two nonzero elements in $\N ^{r}$ and $\N ^{t}$ respectively\footnote{Note that $\gamma _X \notin \ker T_1$ and $\gamma _Y \notin \ker T_2$ because these semigroups are reduced.}, and  the integer matrix
\begin{equation}\label{matriz_pegando}A=\left(
\begin{array}{c|c}
L_1 & 0 \\ \hline
0 & L_2 \\ \hline
\gamma _X & - \gamma _Y
\end{array}
\right).
\end{equation}
Let $S$ be a semigroup such that $\ker S$ is the lattice generated by the rows of the matrix $A.$ This semigroup can be computed by using the Smith Normal Form (see \cite[Chapter 3]{Rosales3}). Denote by $B_1,B_2$ two sets of cardinality $r$ and $t$ respectively satisfying $S=\langle B_1,B_2 \rangle$
and $\ker (\langle B_1,B_2 \rangle)$ is generated by the rows of $A.$

The following Proposition shows that the semigroup $S$ satisfies one of the necessary conditions to be a glued semigroup.

\begin{proposition}\label{proposicion_pegada}
The semigroup $S$ verifies $G(\langle B_1 \rangle)\cap G(\langle B_2 \rangle)=(B_1\gamma _X)\Z=(B_2\gamma _Y)\Z$ with $d=B_1\gamma _X\in \langle B_1 \rangle \cap \langle B_2 \rangle.$
\end{proposition}

\begin{proof}
Use that  $\ker S$ has a basis as (\ref{grupo}) and
proceed as in the proof of the necessary condition of Proposition \ref {proposicion_gluing_gupos}.
\end{proof}

Due to $B_1\cup B_2$ may not  be a minimal generating set, this condition  does not assure that $S$ is a glued semigroup. For instance, taking the numerical semigroups $T_1=\langle 3,5\rangle,$ $T_2=\langle 2,7\rangle$ and $(\gamma _X,\gamma _Y)=(1,0,2,0),$ the matrix obtained from formula (\ref{matriz_pegando}) is
$$\left(
\begin{array}{cc|cc}
 5 & -3 & 0 & 0 \\ \hline
 0 & 0 & 7 & -2 \\ \hline
 1 & 0 & -2 & 0
\end{array}
\right)$$
and $B_1\cup B_2=\{12,20,6,21\}$ is not a minimal generating set.
The following result solve this issue.

\begin{corollary}\label{gamma_afin}
The semigroup
$S$ is a glued semigroup if \begin{equation}\label{condicion_pegada_glued}\sum_{i=1}^r \gamma _{Xi}>1\textrm{ and  }\sum_{i=1}^t \gamma _{Yi}>1.\end{equation}
\end{corollary}

\begin{proof}
Suppose that the set of generators  $B_1\cup B_2$ of $S$ is non-minimal, thus one of its elements is a natural combination of the others. Assume that this element is the first of $B_1\cup B_2$, then there exist $\lambda _2,\ldots ,\lambda _{r+t}\in \N$ such that
$B_1(1,-\lambda _2,\ldots ,-\lambda _r)=B_2(\lambda _{r+1},\ldots ,\lambda _{r+t})\in G(\langle B_1 \rangle)\cap G(\langle B_2 \rangle).$
By Proposition \ref{proposicion_pegada}, there exists $\lambda \in \Z$ satisfying  $B_1(1,-\lambda _2,\ldots ,-\lambda _r)=B_2(\lambda _{r+1},\ldots ,\lambda _{r+t})=B_1(\lambda \gamma _X).$ Since $B_2(\lambda _{r+1},\ldots ,\lambda _{r+t})\in S,$ $\lambda \geq 0$ and thus
$$\nu =(1-\lambda \gamma _{X1},\mathop{\underbrace{
-\lambda _2-\lambda \gamma _{X2},\ldots ,-\lambda _r- \lambda \gamma _{Xr}}}_{\leq 0})\in \ker (\langle B_1 \rangle)=\ker T_1,$$
with the following cases:
\begin{itemize}
\item If $\lambda \gamma _{X1}=0,$ then $T_1$ is not minimally generated which it is not possible by hypothesis.
\item If $\lambda \gamma _{X1}>1,$ then $0>\nu \in \ker T_1,$ but this is not possible due to $T_1$ is a reduced semigroup.
\item If $\lambda \gamma _{X1}=1,$ then $\lambda = \gamma _{X1}=1$ and $$\nu =(0,\mathop{\underbrace{
-\lambda _2- \gamma _{X2},\ldots ,-\lambda _r- \gamma _{Xr}}}_{\leq 0})\in \ker T_1.$$ If $\lambda _i+  \gamma _{Xi}\neq 0$ for some $i=2,\ldots ,r,$ then $T_1$ is not a reduced semigroup.  This implies $\lambda _i=\gamma _{Xi}=0$ for all $i=2,\ldots ,r.$
\end{itemize}
We have just proved that $\gamma _X=(1,0,\ldots ,0).$
In the general case, if $S$ is not minimally generated it is because either $\gamma _X$ or $\gamma _Y$ are elements in the canonical bases of $\N^r$ or $\N ^t,$ respectively. To avoid this situation, it is sufficient to take $\gamma _X$ and $\gamma _Y$ satisfying
$\sum_{i=1}^r \gamma _{Xi}>1$ and $\sum_{i=1}^t \gamma _{Yi}>1.$
\end{proof}

From the above result we obtain a characterization of glued semigroups: $S$ is a glued semigroup if and only if $\ker S$ has a basis as (\ref{grupo}) satisfying Condition (\ref{condicion_pegada_glued}).

\begin{example}\label{ejemplo_torsion}
Let $T_1=\langle (-7,2),(11,1),(5,0),(0,1)\rangle\subset \Z^2$
 and $T_2=\langle 3,	5,	7\rangle\subset \N$ be two reduced affine semigroups.
We compute their associated lattices $$\ker T_1=\langle (1, 2, -3, -4), (2, -1, 5, -3) \rangle \mbox{ and } \ker T_2=\langle (-4, 1, 1), (-7, 0, 3)\rangle.$$ If we take $\gamma _X= (2, 0, 2, 0)$ and $\gamma _Y= (1,2,1),$   the matrix $A$ is
$$\left(\begin{array}{ccccccc}
1 & 2 & -3 & -4 & 0 & 0 & 0 \\
2 & -1 & 5 & -3 & 0 & 0 & 0 \\
0 & 0 & 0 & 0 & -4 & 1 & 1 \\
0 & 0 & 0 & 0 & -7 & 0 & 3 \\
2 & 0 & 2 & 0 & -1 & -2 & -1
\end{array}\right)$$ and the semigroup $S\subset \Z_4 \times \Z^2$ is generated by
$$\{ \mathop{\underbrace{(1, -5, 35), (3, 12, -55), (1, 5, -25), (0, 1, 0)}_{B_1}}, \mathop{\underbrace{(2, 0, 3), (2, 0, 5), (2, 0, 7)}_{B_2}} \}. $$
The semigroup $S$ is the gluing of the semigroups $\langle B_1 \rangle$ and $\langle B_2 \rangle$ and $\ker S$ is generated by the rows of the above matrix.  The ideal $I_S\subset \C [x_1,\ldots ,x_4,y_1,\ldots, y_3]$ is generated\footnote{See \cite{Vigneron} to compute $I_S$ when $S$ has torsion.} by
$$\{x_1x_3^8x_4-x_2^3,\, x_1x_2^2-x_3^3x_4^4,\, x_1^2x_3^5 - x_2x_4^3, x_1^3x_2x_3^2-x_7^7,$$ $$y_1y_3-y_2^2,\, y_1^3y_2-y_3^2,\, y_1^4-y_2y_3,\, \mathop{\underbrace{x_1^{2}x_3^{2} - y_1^{5}y_2}_{\mbox{\footnotesize glued binomial}}} \},$$then $S$ is really a glued semigroup.
\end{example}

\subsection{Generating affine glued semigroups}\label{section_affine}

From Example \ref{ejemplo_torsion} it be can deduced that the semigroup $S$ is not necessarily torsion-free. In general, a semigroup $T$ is affine (or equivalently it is torsion-free) if and only if the {\em invariant factors}\footnote{The invariant factors of a matrix are the diagonal elements of its Smith Normal Form (see \cite[Chapter 2]{Cohen} and \cite[Chapter 2]{Rosales3}).} of the matrix whose rows are a basis of $\ker T$ are equal to one. Assume zero-columns of the Smith Normal Form of a matrix are located on its right side. We now show conditions for $S$ being torsion-free.

Take $L_1$ and $L_2$ the matrices whose rows form a basis of $\ker T_1$ and $\ker T_2$, respectively and
let $P_1,$ $P_2,$ $Q_1$ and $Q_2$ be some matrices with determinant $\pm 1$ (i.e. unimodular matrices) such that $D_1=P_1L_1Q_1$ and $D_2=P_2L_2Q_2$ are the Smith Normal Form of $L_1$ and $L_2,$ respectively.
If $T_1$ and $T_2$ are two affine semigroups, the invariant factors of $L_1$ and  $L_2$ are equal to 1. Then
\begin{equation}\label{smith}\left(
\begin{array}{c|c}
D_1 & 0 \\ \hline
0 & D_2 \\ \hline
\gamma' _X &  \gamma' _Y
\end{array}
\right)
=\left(
\begin{array}{c|c|c}
P_1 & 0 & 0\\ \hline
0 & P_2 & 0\\ \hline
0 & 0 & 1
\end{array}
\right)
\mathop{\underbrace{\left(
\begin{array}{c|c}
L_1 & 0 \\ \hline
0 & L_2 \\ \hline
\gamma _X & - \gamma _Y
\end{array}
\right)}_{=:A}}
\left(
\begin{array}{c|c}
Q_1 & 0 \\ \hline
0 & Q_2
\end{array}
\right),
\end{equation}
where $\gamma ' _X=\gamma _X Q_1$ and  $\gamma ' _Y=-\gamma _Y Q_2.$ Let $s_1$ and $s_2$ be the numbers of zero-columns of $D_1$ and $D_2$  ($s_1, s_2 >0$ because $T_1$ and $T_2$ are reduced,
see \cite[Theorem 3.14]{Rosales3}).

\begin{lemma}\label{mcd_afin}
The semigroup $S$ is an affine semigroup if and only if $$\gcd\Big(\{\gamma'_{Xi}\}_{i=r-s_1}^r \cup \{\gamma'_{Yi}\}_{i=t-s_2}^t\Big)=1.$$
\end{lemma}

\begin{proof}
With the conditions fulfilled by $T_1,$ $T_2$ and $(\gamma _X,\gamma _Y),$ the necessary and sufficient condition for the invariant factors of $A$ to be all equal to one is $\gcd\Big(\{\gamma'_{Xi}\}_{i=r-s_1}^r \cup \{\gamma'_{Yi}\}_{i=t-s_2}^t\Big)=1$.
\end{proof}

The following Corollary  gives the explicit conditions that $\gamma _X$ and $\gamma_Y$ must satisfy to construct an affine semigroup.

\begin{corollary}\label{corolario_afin}
The semigroup $S$ is an affine glued semigroup if and only if:
\begin{enumerate}
\item $T_1$ and $T_2$ are two affine semigroups.
\item $(\gamma _X,\gamma_Y)\in \N ^{r+t}.$
\item $\sum_{i=1}^r \gamma _{Xi},\, \sum_{i=1}^t \gamma _{Yi}>1.$
\item There exist $f_{r-s_1},\ldots ,f_{r},g_{t-s_2},\ldots ,g_{t}\in \Z$ such that
$$(f_{r-s_1},\ldots, f_{r})\cdot (\gamma'_{X{(r-s_1)}},\ldots , \gamma'_{X{r}}) +(g_{t-s_2},\ldots, g_{t})\cdot (\gamma'_{Y{(t-s_2)}},\ldots , \gamma'_{Y{t}})=1.$$
\end{enumerate}
\end{corollary}

\begin{proof}
It is trivial by the given construction, Corollary \ref{gamma_afin} and Lemma \ref{mcd_afin}.
\end{proof}

Therefore, to obtain an affine glued semigroup it is enough to take two affine semigroups and any solution $(\gamma _X,\gamma _Y)$ of the equations of the above corollary.

\begin{example}
Let  $T_1$ and $T_2$ be the semigroups of Example \ref{ejemplo_torsion}.
We compute two elements $\gamma _X=(a_1,a_2,a_3,a_4)$ and $\gamma _Y=(b_1,b_2,b_3)$ in order to obtain an affine semigroup.
First of all, we perform a decomposition of the matrix as (\ref{smith}) by computing the integer Smith Normal Form of $L_1$ and $L_2:$
{\footnotesize $$\left(
\begin{array}{@{\hspace{1pt}}c@{\hspace{6pt}}c@{\hspace{6pt}}c@{\hspace{6pt}}c@{\hspace{6pt}}|c@{\hspace{6pt}}c@{\hspace{6pt}}c@{\hspace{1pt}}}
 1 & 0 & 0 & 0 & 0 & 0 & 0 \\
 0 & 1 & 0 & 0 & 0 & 0 & 0 \\ \hline
 0 & 0 & 0 & 0 & 1 & 0 & 0 \\
 0 & 0 & 0 & 0 & 0 & 1 & 0 \\ \hline
 {a_1} & {a_1}-2 {a_2}-{a_3} & -7 {a_1}+11 {a_2}+5 {a_3} & 2 {a_1}+{a_2}+{a_4} & -{b_1} & {b_1}+2 {b_2}+3 {b_3} & -3 {b_1}-5 {b_2}-7 {b_3}
\end{array}
\right)=$$}
{\small $$\left(
\begin{array}{@{\hspace{1pt}}c@{\hspace{6pt}}c|c@{\hspace{6pt}}c|c@{\hspace{1pt}}}
 1 & 0 & 0 & 0 & 0 \\
 2 & -1 & 0 & 0 & 0 \\ \hline
 0 & 0 & -2 & 1 & 0 \\
 0 & 0 & 7 & -4 & 0 \\ \hline
 0 & 0 & 0 & 0 & 1
\end{array}
\right)\left(
\begin{array}{@{\hspace{1pt}}c@{\hspace{6pt}}c@{\hspace{6pt}}c@{\hspace{6pt}}c|c@{\hspace{6pt}}c@{\hspace{6pt}}c@{\hspace{1pt}}}
 1 & 2 & -3 & -4 & 0 & 0 & 0 \\
 2 & -1 & 5 & -3 & 0 & 0 & 0 \\ \hline
 0 & 0 & 0 & 0 & -4 & 1 & 1 \\
 0 & 0 & 0 & 0 & -7 & 0 & 3 \\ \hline
 {a_1} & {a_2} & {a_3} & {a_4} & -{b_1} & -{b_2} & -{b_3}
\end{array}
\right)\left(
\begin{array}{@{\hspace{1pt}}c@{\hspace{6pt}}c@{\hspace{6pt}}c@{\hspace{6pt}}c|c@{\hspace{6pt}}c@{\hspace{6pt}}c@{\hspace{1pt}}}
 1 & 1 & -7 & 2 & 0 & 0 & 0 \\
 0 & -2 & 11 & 1 & 0 & 0 & 0 \\
 0 & -1 & 5 & 0 & 0 & 0 & 0 \\
 0 & 0 & 0 & 1 & 0 & 0 & 0 \\ \hline
 0 & 0 & 0 & 0 & 1 & -1 & 3 \\
 0 & 0 & 0 & 0 & 0 & -2 & 5 \\
 0 & 0 & 0 & 0 & 0 & -3 & 7
\end{array}
\right)
$$}

Second, by Corollary \ref{corolario_afin}, we must find a solution to the system:
$$\left\{
\begin{array}{l}
a_1+a_2+a_3+a_4>1\\
b_1+b_2+b_3>1\\
f_1,f_2,g_1\in \Z\\
f_1(-7 a_1 + 11 a_2 + 5 a_3)+f_2(2 a_1 + a_2 + a_4) + g_1(-3 b_1 - 5 b_2 - 7 b_3)=1
\end{array}
\right.$$ with $a_1,a_2,a_3,a_4,b_1,b_2 ,b_3\in  \N.$
Such solution is computed (in less than a second) using {\tt FindInstance} of Wolfram Mathematica (see \cite{Mathe}):
$$\mbox{\tt FindInstance}[(-7 {a_1}+11 {a_2}+5 {a_3})*{f_1}+(2 {a_1}+{a_2}+{a_4})*{f_2}+(-3 {b_1}-5 {b_2}-7 {b_3})*{g_1}==1$$ $$\&\&\,{a1}+{a2}+{a3}+{a4}>1\&\&\, {b_1}+{b_2}+{b_3}>1\&\&\, {a_1}\geq 0\&\&\, {a_2}\geq 0\&\&\, {a_3}\geq 0\&\&\, {a_4}\geq 0$$ $$\&\&\, {b_1}\geq 0 \&\&\, {b_2}\geq 0 \&\&\, {b_3}\geq 0,\{{a_1},{a_2},{a_3},{a_4},{b_1},{b_2},{b_3},{f_1},{f_2},{g_1}\},\mbox{ Integers }]$$
$$\downdownarrows$$
$$\{\{{a_1}\to 0,{a_2}\to 0,{a_3}\to 3,{a_4}\to 0,{b_1}\to 1,{b_2}\to 1,{b_3}\to 0,{f_1}\to 1,{f_2}\to 0,{g_1}\to 0\}\}$$

We now take $\gamma _X= (0, 0, 3, 0)$ and $\gamma _Y= (1,1,0),$ and construct the matrix
$$A=\left(\begin{array}{rrrrrrr}
1 & 2 & -3 & -4 & 0 & 0 & 0 \\
2 & -1 & 5 & -3 & 0 & 0 & 0 \\
0 & 0 & 0 & 0 & -4 & 1 & 1 \\
0 & 0 & 0 & 0 & -7 & 0 & 3 \\
0 & 0 & 3 & 0 & -1 & -1 & 0
\end{array}\right).$$
We have the affine semigroup $S\subset \Z^2$ which is minimally generated by
$$\{ \mathop{\underbrace{(2, -56), (1, 88), (0, 40), (1, 0)}_{B_1}}, \mathop{\underbrace{(0, 45), (0, 75), (0,105)}_{B_2}} \} $$ satisfies that $\ker S$ is generated by the rows of $A$ and it is the result of gluing the semigroups $\langle B_1\rangle$ and $ \langle B_2\rangle.$ The ideal $I_S$ is generated by
$$\{x_1x_3^8x_4-x_2^3,\, x_1x_2^2-x_3^3x_4^4,\, x_1^2x_3^5 - x_2x_4^3, x_1^3x_2x_3^2-x_4^7,$$ $$y_1y_3-y_2^2,\, y_1^3y_2-y_3^2,\, y_1^4-y_2y_3,\, \mathop{\underbrace{x_3^{3} - y_1y_2}_{\mbox{\footnotesize glued binomial}}} \},$$therefore, $S$ is a glued semigroup.
\end{example}

All glued semigroups have been computed by using our programm {\tt ecuaciones} which is available in \cite{programa} (this programm requires Wolfram Mathematica 7 or above to run).

\end{document}